\pdfoutput=1 % required by arXiv. Must be in the first 5 lines.

\documentclass[
	11pt,
	a4paper
	]{article}

\usepackage{sty/dl-en}
\usepackage{push_pull_config}
\usepackage{push_pull_macros}

\begin{document}

%----[ title ]----------------------------------------------------------

\maketitle

%----[ footnote ]-------------------------------------------------------

\footnotefirstpage

%----[ No abstract ]----------------------------------------------------

\begin{abstract}
	In this note we present a recipe which transforms
	any pull-push along a span of categories into a push-pull
	along a cospan and vice versa, based on a theorem from Guitart. 
\end{abstract}

\section*{A recipe for base change}
%=======================================================================
%=======================================================================

Consider a lax commutative square
\[
	\begin{tikzcd}[ampersand replacement=\&]
		\category A
		\arrow[r, "s"]
		\arrow[d, "t", swap]
			\& \category B
			\ar[dl, phantom, "\rotatebox{220}{\( \Longrightarrow \)}"]
			\arrow[d, "f"] \\
		\category C
		\arrow[r, "g", swap]
			\&  \category D
	\end{tikzcd}
\]
of small categories and a bicomplete category \( \cosmos \).
The natural transformation
\( fs \implies gt \) induces a \emph{base change natural transformation}
\[
	\PushSum t \Pull s \implies \Pull g \PushSum f
\]
between the pull-push \( \PushSum t \Pull s \) and the
push-pull \( \Pull g \PushSum f \) which both send functors
\( \category B \to \cosmos \) to functors \( \category C \to \cosmos \). 
Dually, using right extensions instead of left extensions, one obtains
another base change transformation
\( \Pull f \PushProduct g \implies \PushProduct s \Pull t \). 

In this note, we address two questions:
\begin{itemize}
	\item
		Given a span of categories, which category should one use
		\[
			\begin{tikzcd}[ampersand replacement=\&]
				\category A
				\arrow[r, "s"]
				\arrow[d, "t", swap]
					\& \category B
					\ar[dl, phantom, "\rotatebox{220}{\(
					\Longrightarrow \)}"] \arrow[d] \\
				\category C
				\arrow[r]
					\& \boxed{?}
			\end{tikzcd}
		\]
		so that the base change transformation be assured
		to be an isomorphism? In other words: is it possible to
		systematically rewrite a pull-push along a span of categories
		\( \category C \leftarrow \category A \to \category B \) as
		a push-pull along a cospan
		\( \category C \to \boxed{?} \leftarrow \category B \)
		of categories? 
	\item
		Given a cospan of categories, which category should one use
		\[
			\begin{tikzcd}[ampersand replacement=\&]
				\boxed{?}
				\arrow[r]
				\arrow[d]
				\&  \category B
				\ar[dl,
				phantom, "\rotatebox{220}{\( \Longrightarrow \)}"]
				\arrow[d, "f"] \\
				\category C
				\arrow[r, "g" swap]
				\& \category D
			\end{tikzcd}
		\]
		so that the base change transformation be assured
		to be an isomorphism? In other words: is it possible to
		systematically rewrite a pull-push along a cospan of categories
		\( \category C \to \category D \leftarrow \category B \) as
		a push-pull along a span
		\( \category C \leftarrow \boxed{?} \to \category B \)
		of categories? 
\end{itemize}

We give an affirmative answer to both questions, using a theorem of
Guitart:
one can always replace the box
\( \boxed{?} \) with the category \emph{universally fitting}
in the lax square diagram.
In the first case, the category universally fitting in the lax square
has seldom been called `co-comma': we shall rename it
\emph{flow sum} for the occasion; it is obtained by
adding arrows \( s(a) \to t(a) \) to the disjoint union
\( \category B \amalg \category C \).
In the second case, the category
universally fitting in the lax square is usually called
`comma category', we rename it \emph{flow product}.

\begin{theorem*}[{\cite[1.14]{zbMATH03779585}}]
	The base change natural transformations
	\[
		\PushSum t \Pull s \implies \Pull g \PushSum f
		\qand
		\Pull f \PushProduct g \implies \PushProduct s \Pull t
	\]
	are isomorphisms when the lax square is
	\begin{itemize}
		\item a flow sum square;
		\item or a flow product square.
	\end{itemize}
\end{theorem*}

Guitart studied base change transformations in the late 1970s;
more recently
his results have been extended to higher categories.
Maltsiniotis gave an extension to derivators
\cite{arXiv:1101.4144}.
Then,
Gepner, Haugseng and Kock showed that the base change
transformation coming from a commutative square of ∞\=/groupoids
is an equivalence \emph{if and only if} the square
is a fibre product
\cite[2.1.6]{arXiv:1712.06469}.

The note is organised in three sections.
In the first section,
we give the definitions of the base change natural transformation,
the flow sum and the flow product.
The second section is dedicated to proving the results, reducing
to the case where \( \category C \) is punctual.
In the last section,
we give a direct application of the base change isomorphism,
describing a context where the pull-push along a span of categories
is functorial.

\section{Base change}
%=======================================================================

\subsection{Pulling and pushing}
%-----------------------------------------------------------------------

Let us fix a bicomplete category \( \cosmos \).
Given a functor \( f \from \category A \to \category B \) between two
small categories,
composition by \( f \) induces a functor
\[
	\begin{tikzcd}
		\fun {\category A} \cosmos
		&
		\fun {\category B} \cosmos
		\lar["\Pull f"']
	\end{tikzcd}
\]
called the \emph{pull-back} functor.
Since \( \cosmos \) is bicomplete the pull-back functor admits two
\emph{push-forward} adjoints.
We shall denote them
\( \PushSum f \adj \Pull f \adj \PushProduct f \).

\subsection{Base change along lax squares}
%-----------------------------------------------------------------------

Consider a lax square of small categories, functors
and natural transformations
\[
	\begin{tikzcd}[ampersand replacement=\&]
		\category A
		\arrow[r, "s"]
		\arrow[d, "t", swap]
		\&  \category B
		\ar[dl, phantom, "\rotatebox{220}{\( \Longrightarrow \)}"]
		\arrow[d, "f"] \\
		\category C
		\arrow[r, "g", swap]
		\&  \category D
	\end{tikzcd}
\]
the natural transformation \( f s \implies g t \)
induces another natural transformation
\( \Pull s \Pull f \implies \Pull t \Pull g \). Using the
unit of the adjunction \( \PushSum f \adj \Pull f  \) and the
counit of the adjunction \( \PushSum t \adj \Pull t \), one obtains
\[
	\PushSum t \Pull s
	\implies \PushSum t \Pull s \Pull f \PushSum f
	\implies \PushSum t \Pull t \Pull g \PushSum f
	\implies \Pull g \PushSum f
\]
a natural transformation
\( \PushSum t \Pull s \implies \Pull g \PushSum f \) that we shall
call the \emph{base change} natural transformation.
In a similar way,
one also obtains a base change formula
\( \Pull f \PushProduct g \implies \PushProduct s \Pull t \).

\begin{remark}
	The natural transformation \( fs \implies gt \) also induces
	a natural transformation
	\( \PushSum g \PushSum t \implies \PushSum f \PushSum s \) which
	in turns gives us an alternative way
	\[
		\PushSum t \Pull s
		\implies \Pull g \PushSum g \PushSum t \Pull s
		\implies \Pull g \PushSum f \PushSum s \Pull s
		\implies \Pull g \PushSum f
	\]
	of writing the base change natural transformation.
\end{remark}

\subsection{Flow sum}
%-----------------------------------------------------------------------

\begin{definition}
	The flow sum of a span of categories
	\( \category C \xleftarrow{t} \category A
	\xrightarrow{s} \category B \),
	is the category
	\( \FlowSum {\category B} {\category A} {\category C} \)
	obtained by enlarging the disjoint union of \( \category B \)
	and \( \category C \)
	\[
		\category B \amalg \category C
		\subset \FlowSum {\category B} {\category A} {\category C}
	\]
	with morphisms \( s(a) \xrightarrow{a} t(a) \)
	for each object \( a \) in \( \category A \) and relations generated
	by morphisms in \( \category A \).

	Concretely,
	\begin{itemize}
		\item
			the set of objects of
			\( \FlowSum {\category B} {\category A} {\category C} \)
			is the disjoint union
			of the set of objects of \( \category B \)
			and the set of objects of \( \category C \);
			\[
				\Objects
				{\FlowSum {\category B} {\category A} {\category C}}
				\coloneqq
				\Objects {\category B}
				\amalg
				\Objects {\category C}
			\]
		\item
			the sets of arrows are as follows. On the one hand, we have
			\begin{align*}
				\relhom{
					\FlowSum
					{\category B}
					{\category A}
					{\category C}
					}{
					b
					}{
					b'
					}
				&\coloneqq \relhom{\category B}{b}{b'}
				\\
				\relhom{
					\FlowSum
					{\category B}
					{\category A}
					{\category C}
					}{
					c
					}{
					c'
					}
				&\coloneqq \relhom{\category C}{c}{c'}
				\\
				\relhom{
					\FlowSum
					{\category B} {\category A} {\category C}
					}{
					c
					}{
					b
					}
				&\coloneqq \emptyset
			\end{align*}
			for any objects \( b,b' \) in \( \category B \)
			and \( c,c' \) in \( \category C \).

			On the other hand
			\[
				\relhom{
					\FlowSum {\category B} {\category A} {\category C}
					}{
					b
					}{
					c}
					\coloneqq \GenSet {\psi a \phi}
			\]
			is generated by words of the
			form \( \psi a \phi \) where \( a \)
			is an object of \( \category A \),
			\( \phi \from b \to s(a) \) is an arrow in \( \category B \)
			and \( \psi \from t(a) \to c \)
			is an arrow in  \( \category C \).
			All theses morphisms are subject to an equivalence relation 
			generated by
			\[
				(\psi' \circ t(\theta)) a \phi
				\sim \psi' a' (s(\theta) \circ \phi)
			\]
			for every morphism \( \theta \from a \to a' \)
			in \( \category A \), every morphism \( \psi' \) 
			in \( \category C \) with source \( t(a') \)
			and every morphism \( \phi \)
			in \( \category B \) with target \( s(a) \).
	\end{itemize}
	Composition of morphisms is given by the composition in
	\( \category B \)
	and the composition in \( \category C \) in the obvious way.
\end{definition}

\begin{figure}
	\[
		\begin{tikzcd}
			b
			\rar["\phi"]
			&
			s(a)
			\rar["a"]
			&
			t(a)
			\rar["\psi"]
			&
			c
		\end{tikzcd}
	\]
	\caption{A morphism from \( b \) to \( c \) in the flow sum.}
\end{figure}

\begin{remark}
	The relation \( \sim \) is not transitive; two three letters
	words \( \psi a \phi \) and \( \psi' a' \phi' \) are thus
	equivalent if and only if there is a commutative hammock
	between them%
	~[\cref{fig: hammock}].
\end{remark}

\begin{figure}
	\[
		\begin{tikzcd}
				& s(a_0) \ar[r,"a_0"] \ar[d,"s(\theta_1)"]
					& t(a_0)
					\ar[d,"t(\theta_1)",swap]
					\ar[ddr, "\psi_0"] \\
				& s(a_1) \ar[r,"a_1",swap]
					& t(a_1) \ar[dr] \\
			b
			\ar[uur, "\phi_0"]
			\ar[ur]
			\ar[ddr,"\phi_n" swap]
				& \vdots \ar[u,"s(\theta_2)",swap]
					& \vdots \ar[u,"t(\theta_2)"]
						& c \\
				& \vdots \ar[d,"s(\theta_n)"] 
					& \vdots \ar[d,"t(\theta_n)",swap] \\
				& s(a_n) \ar[r,"a_n",swap]
					& t(a_n)
					\ar[uur, "\psi_n" swap]
		\end{tikzcd}
	\]
	\caption{A commutative hammock between \( \psi_0 a_0 \phi \)
		and \( \psi_n a_n \phi_n \); they represent the same
		arrow \( b \to c \) in the flow sum.
		}
	\label{fig: hammock}
\end{figure}
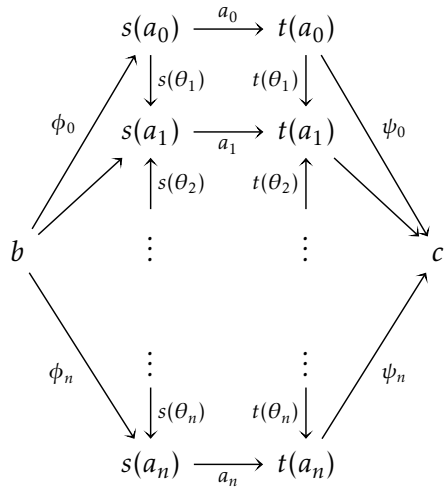

\begin{remark}
\label{remark : flow sum universal}
	The flow sum
	\( \FlowSum {\category B} {\category A} {\category C} \) fits
	in the lax commutative diagram
	\[
		\begin{tikzcd}[ampersand replacement=\&]
			\category A
			\arrow[r, "s"]
			\arrow[d, "t", swap]
			\&  \category B
			\ar[dl, phantom, "\rotatebox{220}{\( \Longrightarrow \)}"]
			\arrow[d] \\
			\category C
			\arrow[r]
			\& \FlowSum {\category B} {\category A} {\category C} 
		\end{tikzcd}
	\]
	and is universal for that property. We shall call such a lax
	square a \emph{flow sum square}.
	Notice that the flow sum is
	in general different from the lax push-out of the same span.
\end{remark}

\subsection{Flow product}
%-----------------------------------------------------------------------

\begin{definition}
	The flow product of a cospan of categories
	\( \category B \xrightarrow{f} \category D \xleftarrow{g}
	\category C \) is the category
	\( \FlowProduct {\category B} {\category D} {\category C} \)
	\begin{itemize}
		\item
			whose objects are the triples \( (b, c, \eta) \)
			with \( b \in  {\category B} \),
			\( c \in  {\category C} \)
			and \( \eta \from f(b) \to g(c) \)
			a map in \( {\category D} \);
		\item
			whose morphisms \( (b,c,\eta) \to (b',c',\eta') \) are pairs
			\( (\phi, \psi) \) where \( \phi \from b \to b' \)
			is a morphism in \( \category B \)
			and \( \psi \from c \to c' \)
			is a morphism in \( \category C \),
			such that the following square
			\[
				\begin{tikzcd}[ampersand replacement=\&]
					f(b)
					\arrow[r, "f(\phi)"]
					\arrow[d, "\eta", swap]
						\& f(b') \arrow[d, "{\eta'}"] \\
					g(c)
					\arrow[r, "g(\psi)", swap]
						\& g(c')
				\end{tikzcd}
			\]
			commutes.
	\end{itemize}
	Composition is given by the formula \( (\phi, \psi) \circ
	(\phi', \psi') = (\phi \circ \phi', \psi \circ \psi') \) for
	two composable morphisms \( (\phi, \psi) \)
	and \( (\phi', \psi') \).
\end{definition}

\begin{remark}
\label{remark : flow product universal}
	The flow product
	\( \FlowProduct {\category B} {\category D} {\category C} \)
	fits in the lax commutative diagram
	\[
		\begin{tikzcd}[ampersand replacement=\&]
			\FlowProduct {\category B} {\category D} {\category C}
			\arrow[r]
			\arrow[d]
				\& \category B
				\ar[dl, phantom,
				"\rotatebox{220}{\( \Longrightarrow \)}"]
				\arrow[d, "f"] \\
			\category C
			\arrow[r, "g", swap]
				\&  \category D
		\end{tikzcd}
	\]
	and is universal for that property.
	We shall call such a lax square a \emph{flow product square}.
	Notice that the flow product
	is in general different from the lax product of the same cospan.
	The flow product is also called the \emph{comma category} and
	is alternatively denoted \( (f,g) \), \( f/g \) or
	\( {\category B} \downarrow_{\category D} {\category C} \) in the
	literature.
\end{remark}

\begin{remark}
	For any two categories \( \category C \) and \( \category B \),
	one can define the category of cospans from \( \category B \)
	to \( \category C \) whose objects are cospans of functors
	\( f \from \category B \to
	\category D \leftarrow \category C \cofrom g \)
	and whose morphisms from \( (f,\category D,g) \)
	to \( (f',\category D',g') \) are functors
	\( h \from \category D \to \category D' \)
	so that \( f' = hf \) and \( g' = hg \).
	One can define in a similar
	way the category \( \Span {\category B} {\category C} \). 
	The construction
	\( \category D \mapsto \FlowProduct {\category B}
	{\category D}{\category C} \)
	may be enhanced into a functor from 
	\( \CoSpan {\category B} {\category C} \)
	to \( \Span {\category B} {\category C} \).
	Similarly, the construction
	\( \category A \mapsto \FlowSum {\category B}
	{\category A}{\category C} \) may be enhanced into
	a functor in the other way.
	The universal properties of the flow sum%
	~[\ref{remark : flow sum universal}]
	and the flow product%
	~[\ref{remark : flow product universal}]
	yield
	\[
		\begin{tikzcd}[ampersand replacement=\&, column sep = 70]
			\Span {\category B} {\category C}
			\arrow[r, shift left=2,"\mathrm{Flow~sum}"]
				\&
				\CoSpan {\category B} {\category C}
				\arrow[l, shift left=2, "\mathrm{Flow~product}"]
		\end{tikzcd}
	\]
	an adjunction between the flow sum functor and the flow product
	functor.
\end{remark}

\begin{remark}
	The fibre product is naturally a full subcategory
	\[
		 \category B \times_{\category D} \category C 
		 	\subset
		 	\FlowProduct {\category B} {\category D} {\category C} 
	\]
	of the flow product: a triple \( (b, c, \eta) \) belongs to
	the fibre product when \( \eta \) is an identity morphism.
\end{remark}

\subsection{Base change theorem}
%-----------------------------------------------------------------------

\begin{theorem}
	Given a lax commutative square
	\[
		\begin{tikzcd}[ampersand replacement=\&]
			\category A
			\arrow[r, "s"]
			\arrow[d, "t", swap]
				\& \category B
				\ar[dl, phantom, "\rotatebox{220}{\(
				\Longrightarrow \)}"]
				\arrow[d, "f"] \\
			\category C
			\arrow[r, "g", swap]
				\&  \category D
		\end{tikzcd}
	\]
	of small categories, 
	the base change natural transformations
	\[
		 \PushSum t \Pull s \implies \Pull g \PushSum f 
		 	\qand
		 	\Pull f \PushProduct g \implies \PushProduct s \Pull t 
	\]
	are isomorphisms when the above square is either
	\begin{itemize}
		\item
			a flow sum square;
		\item
			or a flow product square.
	\end{itemize}
\end{theorem}

\begin{proof}
	The proof that
	\( \PushSum t \Pull s \implies \Pull g \PushSum f \)
	is an isomorphism can be deduced from the punctual case
	for both the flow product case%
	~[\ref{thm: flow product case}]
	and the flow sum case%
	~[\ref{thm: flow sum case}]; it is the subject of the next section.
	For \( \Pull f \PushProduct g \implies \PushProduct s \Pull t \),
	one only needs to change \( \cosmos \) into \( \cosmos\op \). 
\end{proof}

\begin{remark}
	In the literature, the base change isomorphism
	is also called the `projection formula'
	or the `Beck-Chevalley condition'.
	Squares inducing a
	base change isomorphism are called \emph{exact} by Guitart.
\end{remark}

\section{The proof}
%=======================================================================

\subsection{Composition of base change transformations}
%-----------------------------------------------------------------------

Given two adjacent squares
\[
	\begin{tikzcd}
		\category E \rar["k"] \dar["l" swap]
			& \category A \rar["s"] \dar["t"]
			\ar[dl, phantom, "\rotatebox{220}{\( \Longrightarrow \)}"]
				& \category B \dar["f"]
				\ar[dl, phantom, "\rotatebox{220}{
				\( \Longrightarrow \)}"] \\
		\category F \rar["r" swap]
			& \category C \rar["g" swap]
				& \category D
	\end{tikzcd}
\]
one can paste the two natural transformations
\( (fs \implies gt)k \) and \( g(tk \implies rl) \) 
into a natural transformation from \( fsk \) to \( grl \).

\begin{lemma}[{\cite[2.2]{doi:10.1007/Bfb0063101}}]
	\label{thm: composition of base change natural transformations}
	Base change natural transformations can be composed along
	adjacent squares: given two adjacent squares
	the base change transformation
	\( \PushSum l \Pull {(sk)} \implies \Pull {(gr)} \PushSum f \) 
	is the composite of
	\( (\PushSum l \Pull k \implies \Pull r \PushSum t) \Pull s \)
	and \( \Pull r (\PushSum t \Pull s \implies \Pull g \PushSum f) \).
\end{lemma}

\subsection{Composition of flow products}
%-----------------------------------------------------------------------

Fibre products of categories can be composed along adjacent squares:
given a diagram of small categories,
\[
	\begin{tikzcd}
		\category E \rar \dar
			& \category A \rar \dar
				& \category B \dar \\
		\category F \rar
			& \category C \rar
				& \category D
	\end{tikzcd}
\]
if the left and right squares are both fibre products, then the
external rectangle is also a fibre product. This can be
summarised
\[
	\left( \category B \times_{\category D} \category C \right)
	\times_{\category C} \category F
		\isonat \category B \times_{\category D} \category F
\]
in a short formula. 

Similarly, flow products can be composed with fibre products, as
explained in the following straightforward lemma.

\begin{lemma}
	\label{thm: fibre product and flow product}
	If the right square is a flow product square, then the
	external rectangle is a flow product square if and only if
	the left square is a fibre product square.
	In particular, one has
	\[
		\left(
		\FlowProduct {\category B} {\category D} {\category C}
		\right)
		\times_{\category C} \category F
			\isonat
			\FlowProduct {\category B} {\category D} {\category F}
	\]
	in a short formula.
\end{lemma}

\begin{remark}
	\label{rmk: fibre product and flow product}
	One has the same result
	\[
		{\category F} \times_{\category B}
		\left(
			\FlowProduct {\category B} {\category D} {\category C}
		\right)
			\isonat
			\FlowProduct {\category F} {\category D} {\category C}
	\]
	if one would have stacked
	the fibre product square above the flow product square.
	Similarly, one has a dual result with
	flow sums and the disjoint union of small categories.
\end{remark}

\subsection{Pointwise base change}
%-----------------------------------------------------------------------

\begin{definition}
	Let \( f \from \category B \to \category D \) be a functor
	between two small categories and let \( d \) be an object of
	\( \category B \).
	The \emph{fibre} \( f^{-1}(d) \)
	of \( f \) at \( d \)
	is the fibre product 
	\[
		\begin{tikzcd}[ampersand replacement=\&]
			f^{-1}(d)
			\arrow[r, ""]
			\arrow[d, "",swap]
			\arrow[rd, very near start, phantom, "\lrcorner"]
			\& \category B
			\arrow[d, "f"] \\
			\ast
			\arrow[r, "d", swap]
			\& \category D
		\end{tikzcd}
	\]
	while the \emph{flow} \( \Flow f d \) of \( f \) to \( d \) is
	\[
		\begin{tikzcd}[ampersand replacement=\&, column sep = large]
			\Flow f d
			\arrow[r]
			\arrow[d]
			\arrow[rd, very near start, phantom, "\lrcorner"]
			\& \category B
			\ar[dl, phantom, "\rotatebox{220}{\( \Longrightarrow \)}"]
			\arrow[d, "f"] \\
			\ast
			\arrow[r, "d", swap]
			\&  \category D
		\end{tikzcd}
	\]
	the flow product of the same cospan.
	Notice that by definition,
	the fibre \( f^{-1}(d) \) is a full subcategory of the flow
	\( \Flow f d \).
\end{definition}

\begin{remark}
	The flow of \( f \) to \( d \)  is also denoted as a slice
	\( \slice f d \) in the literature.
\end{remark}

\begin{definition}
	A functor \( f \from \category A \to \category B \) is said to be
	cofinal
	if any diagram of shape \( \category B \) valued in any category
	$\category W$ can be
	restricted to a diagram of shape \( \category A \)
	along \( f \) without changing its colimit.
	This happens exactly when the flow
	\( \FlowProduct d {\category D} {\category B} \) of \( f \) from
	\( d \) 
	is connected, for every object \( d \in \category D \).

	When \( f \) is cofinal, the canonical map
	\[
		\PushSum \proj \PushSum f \Pull f \implies \PushSum \proj
	\]
	is an isomorphism, where \( \proj \from \category B \to \ast \) is
	the canonical map to the punctual category.
\end{definition}

\begin{definition}
	Let \( f \from \category A \to \category B \) be a functor.
	An arrow \( \phi \from a \to a' \) in \( \category A \) is
	\( f \)\=/cocartesian if for every other arrow
	\( \phi' \from a \to a'' \) such that
	\( f(\phi') = \psi \circ f(\phi) \), there exists a unique
	lift \( \widetilde \psi \) of \( \psi \) such that
	\( \phi' = \widetilde \psi \circ \phi \). 

	The functor \( f \) is called an \emph{opfibration} if for
	every \( a \) in \( \category A \) and every
	\( \phi \from f(a) \to b \), there exists a \( f \)\=/cocartesian
	lifting \( \widetilde\phi \from a \to a' \) of \( \phi \).
\end{definition}

\begin{lemma}
	\label{thm: cofinality for opfibrations}
	When \( f \from \category B \to \category D \) is an
	opfibration, the inclusion
	\( f^{-1}(d) \subset \Flow f d \) is cofinal.
\end{lemma}

\begin{proof}
	Let \( (b, \eta \from f(b) \to d) \) be an object
	of the flow of \( f \) to \( d \).
	We shall show that the flow of the inclusion from \( (b, \eta) \)
	(that is, it's coslice by \( (b, \eta) \))
	is connected.
	Since \( f \) is an opfibration, \( \eta \) admits
	a cocartesian
	lift \( \theta \from b \to e \) and we have thus a morphism
	\( \theta \from (b, \eta) \to (e, \id_d) \); it is initial
	and the flow from \( (b, \eta) \) is thus connected.
\end{proof}

\begin{lemma}
	\label{thm: point case}
	Consider a lax square
		\[
			\begin{tikzcd}[ampersand replacement=\&]
				\category A
				\arrow[r]
				\arrow[d]
					\&  \category B
					\ar[dl, phantom, "\rotatebox{220}{
					\( \Longrightarrow \)
					}"] \arrow[d, "f"] \\
				\ast
				\arrow[r, "d"']
					\& \category D
			\end{tikzcd}
		\]
	of small categories. The base change natural transformation is
	an isomorphism when either
	\begin{itemize}
		\item
			\( \category A \) is
			the flow \( \Flow f d \) of \( f \)
			to \( d \in \category D \);
		\item
			or \( \category A \) is the fibre \( f^{-1}(d) \) of
			\( f \) at
			\( d \in \category D \) and \( f \) is an opfibration.
	\end{itemize}
\end{lemma}

\begin{proof}
	The first case comes from the punctual evaluation of
	\( \PushSum f \) 
	\[
		\PushSum f F(d)
		\isonat \textstyle\colimover{(b, f(b) \to d)} F(b)
	\]
	where \( F \) is any functor from \( \category B \) to
	\( \cosmos \);
	this is a classic result of category theory and can
	be found for example in
	\emph{Categories for the Working Mathematician}
	\cite[X.5.3]{doi:10.1007/978-1-4757-4721-8}.
	The second case follows from the first case by cofinality%
	~[\ref{thm: cofinality for opfibrations}]
	and composition of base change natural transformations%
	~[\ref{thm: composition of base change natural transformations}].
\end{proof}

\subsection{Flow product case}
%-----------------------------------------------------------------------

We are now ready to give the proof of the general case for the
flow product. Consider the following
\[
	\begin{tikzcd}[ampersand replacement=\&]
		\FlowProduct {\category B} {\category D} {\category C}
		\arrow[r, "s"]
		\arrow[d, "t" swap]
		\arrow[rd, very near start, phantom, "\lrcorner"]
			\& \category B
			\ar[dl, phantom, "\rotatebox{220}{\( \Longrightarrow \)}"]
			\arrow[d, "f"] \\
		\category C
		\arrow[r, "g", swap]
			\&  \category D
	\end{tikzcd}
\]
flow product square.

\begin{lemma}
	The projection \( t \from
	\FlowProduct {\category B} {\category D} {\category C}
	\to {\category C} \)
	is an opfibration.
\end{lemma}

\begin{proof}
	For any object \( (b, c, \eta \from f(b) \to g(c)) \)
	in \( \FlowProduct {\category B} {\category D} {\category C} \)
	and any map
	\( \phi \from c \to c' \) in \( \category C \),
	a cocartesian lifting of \( \phi \)
	is given by the map
	\( (\id_b, \phi) \from (b, c, \eta)
	\to (b, c', g(\phi) \circ \eta) \).
\end{proof}

\begin{proposition}%
	\label{thm: flow product case}
	In the case of a flow product square as above, the base change
	natural transformation is an isomorphism.
\end{proposition}

\begin{proof}
	In order to show that the base change natural transformation is
	an isomorphism, it is enough to show that it gives an isomorphism
	at each \( c \) in \( \category C \). By pulling-back \( t \)
	along \( c \from \ast \to \category C \)
	\[
		\begin{tikzcd}[ampersand replacement=\&, column sep = large]
			\Flow f {g(c)}
			\ar[r,"u"]
			\ar[d,"v" swap]
			\arrow[rd, very near start, phantom, "\lrcorner"]
				\&
				\FlowProduct {\category B} {\category D} {\category C}
				\arrow[rd, very near start, phantom, "\lrcorner"]
				\ar[r,"s"]
				\ar[d,"t" swap]
					\& \category B \ar[d,"f"]
					\ar[
						dl,
						phantom,
						"\rotatebox{220}{\( \Longrightarrow \)}"
					] \\
			\ast \ar[r,"c" swap]
				\& \category C \ar[r, "g"']
					\& \category D
		\end{tikzcd}
	\]
	one obtains an outer rectangle which is a flow product square%
	~[\ref{thm: fibre product and flow product}].
	As a consequence and since \( t \) is an opfibration,
	the base change natural transformations%
	\[
		\PushSum v \Pull u \implies \Pull c \PushSum t
		\qand
		\PushSum v \Pull u \Pull s \implies \Pull c \Pull g \PushSum f 
	\]
	are isomorphisms
	~[\ref{thm: point case}].
	Then using the composition of base change natural transformations%
	~[\ref{thm: composition of base change natural transformations}]
	and the 2-out-of-3 property of isomorphisms
	\[
		\Pull c \PushSum t \Pull s \implies \Pull c \Pull g \PushSum f
	\]
	is also an isomorphism as required.
\end{proof}
	
When \( f \) is an opfibration, one can replace the
flow product by the fibre product in order for the
base change to be an isomorphism.

\begin{proposition}[{\cite[11.6]{Joyal}}]
	\label{thm: opfibration case}
	Let
	\[
		\begin{tikzcd}[ampersand replacement=\&]
			\category C \times_{\category D} \category B
			\arrow[r, ""]
			\arrow[d, "",swap]
			\arrow[rd, very near start, phantom, "\lrcorner"]
			\& \category B
			\arrow[d, "f"] \\
			\category C
			\arrow[r, "g", swap]
			\& \category D
		\end{tikzcd}
	\]
	be a fibre product of small categories and assume \( f \) is
	an opfibration, then the base change natural transformation
	is an isomorphism.
\end{proposition}

\begin{proof}
	Following the same line of thoughts, we pull-back along
	\( c \from \ast \to \category C \) and end up with
	\[
		\begin{tikzcd}
			f^{-1}(g(c))
			\arrow[rd, very near start, phantom, "\lrcorner"]
			\dar \rar
				& \category C \times_{\category D} \category B
				\arrow[rd, very near start, phantom, "\lrcorner"]
				\dar \rar
					& \category B \dar["f"] \\
			\ast
			\rar["c" swap]
				& \category C
				\rar["g" swap]
					& \category D
		\end{tikzcd}
	\]
	a pull-back rectangle and since the pull-back of an opfibration
	is again an opfibration, we can conclude using the punctual
	case.
\end{proof}

\subsection{Flow sum case}

As in the flow product case, we shall show that the base change
formula is a point-wise isomorphism by pulling back along each
object \( c \from \ast \to \category C \).

Given a lax square
\[
	\begin{tikzcd}[ampersand replacement=\&]
		\category A
		\arrow[r, "s"]
		\arrow[d, "t", swap]
			\&  \category B
			\ar[dl, phantom, "\rotatebox{220}{\( \Longrightarrow \)}"]
			\arrow[d, "f"] \\
		\category C
		\arrow[r, "g", swap]
			\&  \category D
	\end{tikzcd}
\]
and an object \( c \in \category C \), the functor \( s \) induces
a functor
\[
	\begin{tikzcd}
		\Flow t c \rar["s"] & \Flow f {g(c)} 
	\end{tikzcd}
\]
sending a pair \( (a, \phi \from t(a) \to c) \)
to the pair
\( (s(a), fs(a) \to gt(a) \xrightarrow{g(\phi)} g(c)) \).

\begin{lemma}
	\label{thm : cofinal functor flow sum}
	If \( \category D \) is the flow sum of the span
	\( \category C \leftarrow \category A \to \category B \), then
	the functor \( s \from \Flow t c \to \Flow f {g(c)} \) is cofinal
	for every \( c \in \category C \). 
\end{lemma}

\begin{proof}
	Let \( b \) be an object of \( \category B \) and
	\( \psi a \phi \) be an arrow from \( b \) to \( c \) in
	the flow sum. We shall show that the flow of \( s \) from
	\( (b, \psi a \phi) \) is connected. 

	It is not empty since there is a morphism
	\( \phi \from (b, \psi a \phi) \to (s(a), \psi a) \). 
	Let \( \phi' \from (b, \psi a \phi) \to (s(a'), \psi'a') \)
	be another object of the flow.
	Since the two morphisms \( \psi a \phi \) and \( \psi' a' \phi' \)
	are equal in the flow sum,
	they are connected by a commutative hammock diagram%
	~[\cref{fig: hammock}].
	As a consequence, any two objects of the flow can be connected via
	a zigzag coming from a hammock diagram: the flow is connected.
\end{proof}

\begin{proposition}%
	\label{thm: flow sum case}
	If the lax square above is a flow sum square,
	the base change natural transformation is an isomorphism.
\end{proposition}

\begin{proof}
	We want to show that the natural transformation
	\[
		\Pull c \PushSum t \Pull s \implies \Pull c \Pull g \PushSum f
	\]
	is an isomorphism for every \( c \in \category C \).
	For this, let us consider this cube
	\[
		\begin{tikzcd}[column sep = {60,between origins},
			row sep = {40,between origins}]
			\Flow t c
			\ar[rr]
			\ar[dd]
			\ar[rd]
				&
					& \Flow f {g(c)} \ar[rd] \ar[dd]
						& \\
				& \category A \ar[rr, crossing over]
					&
						& \category B \ar[dd, "f"] \\
			\ast
			\ar[rd, "c" swap]
			\ar[rr, equal]
				&
					& \ast \ar[rd, "g(c)"]
						& \\
				& \category C \ar[rr, "g" swap]
				\ar[from=uu, crossing over]
					&
						& \FlowSum {\category B} {\category A}
						{\category C}
		\end{tikzcd}
	\]
	whose top and bottom faces are commutative and whose other faces are
	lax commutative.
	
	Since base change for the left face is an isomorphism%
	~[\ref{thm: point case}],
	by the 2-out-of-3
	property of isomorphisms and the composition of base change natural
	transformations%
	~[\ref{thm: composition of base change natural transformations}],
	we need to show that the base change natural
	transformation is an isomorphism for the front rectangle made of
	the left face and the front face.
	
	Now, this rectangle is equal to
	the back rectangle made of the back face and the right face;
	by the previous lemma ~[\ref{thm : cofinal functor flow sum}],
	base change is an isomorphism for
	the back face and we know that it is also the case for the right
	face%
	~[\ref{thm: point case}],
	so it is an isomorphism for the back rectangle.
\end{proof}

\section{Functorial pull-push}
%=====================================================================

We give here a direct application of the base change isomorphism for
flow product squares and flow sum squares.

\begin{remark}
	Given two cospans
	\( \category A \to \category B \leftarrow \category C \)
	and 
	\( \category C \to \category D \leftarrow \category E \),
	using previous computations%
	~[\ref{thm: fibre product and flow product},
	\ref{rmk: fibre product and flow product}]
	we get canonical isomorphisms
	\[
		\FlowProduct {\category A} {\category B}
		{\left(
		\FlowProduct {\category C} {\category D} {\category E}\right)}
			\isonat
			\left(
				\FlowProduct {\category A} {\category B} {\category C}
			\right)
			\times_{\category C}
			\left(
				\FlowProduct {\category C} {\category D} {\category X}
			\right)
				\isonat
				\FlowProduct
				{\left(
				\FlowProduct {\category A} {\category B} {\category C}
				\right)}
				{\category D} {\category E}
	\]
	or in other words: the flow product is associative.
	Similarly, flow sums are also associative.
\end{remark}

\begin{definition}
	Let \( \CatSpan \) be the pseudo-category
	whose objects are the small
	categories and whose morphisms between two categories
	\( \category A \) and \( \category C \)
	are the spans
	\( \category A \leftarrow \category B \to \category C \). 
	Composition of 
	\( \category A \leftarrow \category B \to \category C \)
	with
	\( \category C \leftarrow \category D \to \category E \) is given
	by the span
	\( \category A \leftarrow
	\FlowProduct {\category B} {\category C} {\category D}
	\to \category E \).
\end{definition}

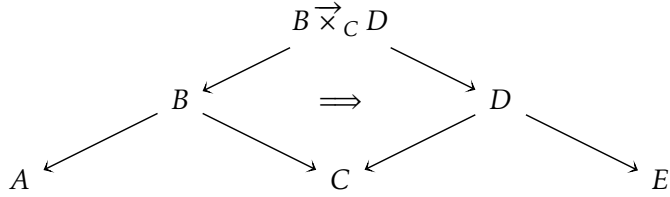
\begin{figure}
	\[
		\begin{tikzcd}[column sep = {60,between origins},
			row sep = {30,between origins}]
			& &
			\FlowProduct {\category B} {\category C} {\category D}
			\ar[dl] \ar[dr]
			& &
			\\
			& \category B
			\ar[dl] \ar[dr]
			& \Longrightarrow &
			\category D
			\ar[dl] \ar[dr]
			&
			\\
			\category A & &  \category C & & \category E
		\end{tikzcd}
	\]
	\caption{Composition of spans of categories.}
\end{figure}

\begin{remark}
	The associativity of the flow product
	insures that the composition of spans given
	in the definition is associative; it does not have a unit though.
	Using the associativity of flow sums, one could instead build
	a category whose morphisms are given by cospans.
\end{remark}

The assignment
	\[
		\left(\category A \xleftarrow{f} \category B
		\xrightarrow{g} \category C\right)
			\longmapsto
			\left(\fun {\category A} \cosmos
			\xrightarrow{\PushSum g \Pull f}
			\fun {\category C} \cosmos\right)
	\]
gives us a functor
	\[
		\begin{tikzcd}[column sep = 70]
			\CatSpan
			\rar["\category A \mapsto \fun {\category A} \cosmos"]
				& \BigCats
		\end{tikzcd}
	\]
from the pseudo-category of small categories and spans
to the (2,1)\=/category of categories,
functors and natural isomorphisms.

The same thing can be done using cospans as morphisms, in which
case one obtains a similar theorem.
One could also change the formula
\( \PushSum g \Pull f \) to \( \PushProduct g \Pull f \).

\subsection*{Acknowledgements}
%-----------------------------------------------------------------------

Damien Lejay was supported by \IBS.
The authors would like to thank Rune Haugseng for useful discussions
and Denis-Charles Cisinski for pointing us to the articles of
Guitart and Maltsiniotis.

%====[ Bibliography ]===================================================

\bibliography{bib/dl}

\providecommand{\href}[2]{#2}\begingroup\raggedright\begin{thebibliography}{1}

\bibitem{zbMATH03779585}
René Guitart, `Relations et carrés exacts', {\em Annales des Sciences
  Mathématiques du Québec} {\bfseries 4} (1980) 103--125.

\bibitem{arXiv:1101.4144}
Georges Maltsiniotis, `Carrés exacts homotopiques, et dérivateurs', 2011.
\newblock \href{http://arxiv.org/abs/1101.4144}{{\color{arXiv}\ttfamily
  arXiv:1101.4144 [math.AT]}}.

\bibitem{arXiv:1712.06469}
David Gepner, Rune Haugseng, and Joachim Kock, `∞-Operads as analytic
  monads', 2017.
\newblock \href{http://arxiv.org/abs/1712.06469}{{\color{arXiv}\ttfamily
  arXiv:1712.06469 [math.AT]}}.

\bibitem{doi:10.1007/Bfb0063101}
G.~M. Kelly and Ross Street,
  \href{http://dx.doi.org/10.1007/BFb0063101}{`Review of the elements of
  2-categories',} in {\em Category Seminar}, Gregory~M. Kelly, ed.,
  pp.~75--103.
\newblock Springer Berlin Heidelberg, Berlin, Heidelberg, 1974.

\bibitem{doi:10.1007/978-1-4757-4721-8}
Saunders Mac~Lane, \href{http://dx.doi.org/10.1007/978-1-4757-4721-8}{{\em
  Categories for the Working Mathematician}}, vol.~5 of {\em Graduate Texts in
  mathematics}.
\newblock Springer New York, New York, NY, 1978.

\bibitem{Joyal}
André Joyal, `The theory of quasi-categories and its applications', 2008.
\newblock Lecture notes.

\end{thebibliography}\endgroup

\end{document}